\documentclass[MainFile.tex]{subfiles}
\newtheorem{zeorem}{Theorem}[section]

\newtheorem{zemma}[zeorem]{Lemma}
\newtheorem{zorollary}[zeorem]{Corollary}

\newcommand\blfootnote[1]{%
  \begingroup
  \renewcommand\thefootnote{}\footnote{#1}%
  \addtocounter{footnote}{-1}%
  \endgroup
}
\begin{document}
\title{Kesten's theorem for uniformly recurrent subgroups}
\author{Mikolaj Fraczyk}
\address{Alfr\'ed R\'enyi Institute of Mathematics,  Re\'altanoda utca 13-15, H-1053, Budapest, Hungary}
  \email{fraczyk@renyi.hu}
\date{}
\maketitle
\blfootnote{\today}
\begin{abstract}
We prove an inequality on the difference between the spectral radius of the Cayley graph of a group $G$ and the spectral radius of the Schreier graph $H\backslash G$ for any subgroup $H$. As an application we extend Kesten's theorem on spectral radii to uniformly recurrent subgroups and give a short proof that the result of Lyons and Peres on cycle density in Ramanujan graphs \cite[Theorem 1.2]{LyonsPeres} holds on average. More precisely, we show that if $\mathcal G$ is an infinite deterministic Ramanujan graph then the time spent in short cycles by a random walk of length $n$ is $o(n)$. 
\end{abstract}
\section{Introduction}
\subsection{Kesten theorems}
The spectral radius of a $d$-regular, countable, undirected graph $\mathcal G$ is the operator norm of the Markov averaging operator $P:\ell^2(\mathcal G)\to\ell^2(\mathcal G)$ defined as $Pf(v)=\frac{1}{\deg v}\sum_{v'\sim v}f(v')$. We denote it by $\rho(\mathcal G)$. For connected graphs there is another interpretation of spectral radius via return probabilities. Fix an origin $o$ of $\mathcal G$ and write $A_{\mathcal G}(n)$ for the set of walks starting at $o$ and returning to $o$ after time $n$. Then, the spectral radius of $\mathcal G$ is the limit 
$$\rho(\mathcal G)=\lim_{n\to\infty}\left(\frac{|A_{\mathcal G}(2n)|}{d^{2n}}\right)^{1/2n}.$$ 
Let $G$ be a countable group generated by a symmetric set $S$ and let $H$ be a sugbroup of $G$. Write $\Cay(G,S)$ for the Cayley graph and $\Sch(H\backslash G,S)$ for the Schreier graph on the left $H$-cosets. Once $S$ is fixed we will write $\rho(G)=\rho(\Cay(G,S))$ and $\rho(H\backslash G)=\rho(H\backslash G,S)$. This paper expands on the criteria for amenability given by Kesten \cite{Kesten1, Kesten2}
\begin{zeorem}[Kesten]\label{Kesten1}
Let $G$ be a group generated by a finite symmetric set $S$. Then $G$ is amenable if and only if $\rho(G)=1$. 
\end{zeorem}
\begin{zeorem}[Kesten]\label{Kesten2}
Let $G$ be a group generated by a finite symmetric set $S$ and let $H$ be a normal subgroup of $G$. Then $H$ is amenable if and only if $\rho(G)=\rho(H\backslash G)$. 
\end{zeorem}
Note that Theorem \ref{Kesten1} follows from Theorem \ref{Kesten2} applied to $H=G$. If $H$ is amenable then $\rho(G)=\rho(H\backslash G)$ holds unconditionally but the converse is not true in general. We shall say that a subgroup $H$ of $G$ is a \textbf{Ramanujan subgroup} with respect to $S$ if $\rho(G)=\rho(H\backslash G)$. The space $\Sub_G$ of subgroups of a locally compact group $G$ is endowed with a natural topology called the Chabouty topology \cite{dlHarpe1}. An \textbf{invariant random subgroup}, IRS for short, is a $\Sub_G$ valued random variable whose ditribution is invarinat under conjugation by $G$. In \cite{AGV14} Abert, Glasner and Virag proved a probabilistic version of Kesten's theorem for IRS'ses: 
\begin{zeorem}[Abert,Glasner,Virag]\label{KestenIRS}
Let $G$ be a group generated by a finite symmetric set $S$ and let $H$ be an invariant random subgroup of $G$. Then $H$ is amenable almost surely if and only if $H$ is Ramanujan almost surely.
\end{zeorem} 
In other words an IRS is Ramanujan if and only if it is amenable. We refer to the article \cite{AGV14} for other properties of IRS'ses. In the present paper we develop a quantitative version of Kesten's theorem that works for for any (deterministic) subgroup $H$ of $G$. Our main technical result (Theorem \ref{t.SpecIneq}) is the inequality
\begin{equation}\label{MainIneq}
\log \rho(H\backslash G,S)-\log\rho(G,S)\geq \frac{\limsup_{n\to\infty}\int |H^g\cap S|(-\log\rho(H^g,S\cap H^g))d\mu_{2n}(g)}{|S|^2\rho(G,S)^2},
\end{equation} where $\mu_{2n}$ is averaging measure over traces of recurrent random walks of length $2n$ (for definition see equation \ref{d-tracemeasure}). 

\subsection{Uniformly recurrent subgroups} 
Uniforlmy recurrent subgroup is a toplogical analogue of an ergodic IRS.  Write ${\rm Sub}_{G}$ for the set of subgroups of a group $G$ endowed with the Chabauty topology \cite{dlHarpe1}. Let $H$ be a subgroup of $G$ and let $X$ be the closure of the congujacy class of $H$ in ${\rm Sub}_G$. Subgroup $H$ is called a \textbf{ uniformly recurrent subgroup } (URS for short) if $X$ is minimal as a dynamical $G$-system i.e. every $G$-orbit in $X$ is dense. The notion of URS was introduced by Glasner and Weiss in \cite{URS} and was further studied in recent papers \cite{Elek, MatTod}. Uniformly recurrent subgroups atracted some attention since Kennedy \cite{Kennedy} proved that a countable group $G$ is $C^*$-simple if and only if it has no amenable URS.

We prove that the natural extension of Kesten's theorem holds for URS'ses.
\begin{zeorem}\label{ThmA}
Let $G$ a countable group generated by a finite set $S$ and let $H$ be an URS of $G$. Then $H$ is amenable if and only if $\rho(G)=\rho(H\backslash G)$.
\end{zeorem}
It was already shown in \cite{URS} that an URS $H$ amenable if and only if every subgroup in $X$ is. Similarly an URS $H$ is Ramanujan if and only if $X$ contains only Ramanujan subgroups.
\subsection{Cycle density in Ramanujan graphs}
Let $(G,x)$ be a $d$-regular Ramanujan graph and let $k\geq 1$ be a fixed integer. For any $n\geq 0$ write $q_n$ for the probability that a random walk starting at $x$ lies at time $n$ on a cycle of length at most $k$. In \cite{LyonsPeres} Lyons and Peres proved that $\lim_{n\to\infty}q_n=0$. Their result was motivated by \cite[Problem 11]{AGV12}. Using inequality \ref{MainIneq} we show (Theorem \ref{T1S3}) that
$$\lim_{n\to\infty}\frac{1}{n}\sum_{j=1}^n q_j=0.$$
In other words the random walks on a Ramanujan graph do not spend much time in the short cycles. This gives a relatively simple proof that the conclusion of \cite[Theorem 1.2]{LyonsPeres} holds on average.
\subsection{Organisation of the paper.} Section \ref{S1} is devoted to the proof of (\ref{t.SpecIneq}). We follow closely the argument from \cite{AGV14}. Our main contribution is an application of the inequality between the arithmetic and the geometric means at certain point of the proof, similarly to how it was used in \cite{AGV12}. 
In Section \ref{S2}we apply \ref{t.SpecIneq} to prove Theorem \ref{ThmA}. In section \ref{S3} we give a relatively short proof that  the conclusion of \cite[Theorem 1.2]{LyonsPeres} holds on average.
\subsection*{Acknowledgment.} Part of this work was included in the author's master thesis supervised by Emmanuel Breuillard whom I thank for suggesting the topic. I would also like to thank Miklos Abert for bringing my attention to the work of Russel Lyons and Yuval Peres. The author was partially supported by the ERC Consolidator Grant No. 648017 and partially by a public grant ANR-11-LABX-0056-LMH.
\section{Inequality on the spectral radii}\label{S1}
Let $G$ be a group generated by a finite symmetric set $S$. Recall that for a subgroup $H$  of $G$ we write $\Sch(H\backslash G,S)$ for the Schreier graph encoding the action of generators from $S$ on the coset space  $H\backslash G$. Group $G$ acts from the right on $\Sch(H\backslash G,S)$. Write $A(n,S),A_H(S,n)$ for the sets of walks on $\Cay(G,S)$ respectively $\Sch(H\backslash G,S)$ that return to the identity after $n$ steps. We identify the walks of length $n$ with sequences in $S^n$. For a walk $w=(a_1,\ldots,a_n)$ we write $w(i)=a_1a_2\ldots a_i$ for the position after $i$ steps. If $H$ is a subgroup of $G$ and $F\subset H$ is a finite subset we adopt the convention that $\rho(H,F)=\rho(\langle H\cap F\rangle,F)$ in the case when $F$ does not generate $H$. We will write $H^g=g^{-1}Hg$.

Define the measure
\begin{equation}\label{d-tracemeasure} \mu_n=\frac{1}{|A(n,S)|}\sum_{w\in A(n,S)}\left(\frac{1}{n}\sum_{i=1}^n \delta_{w(i)}\right).\end{equation}
Intuitively $\mu_n(g)$ is the expected proportion of time spent in $g$ by a random recurrent walk of length $n$.  
We have 
\begin{zeorem}\label{t.SpecIneq} Let $I(H,S)=\limsup_{n\to\infty}\int |H^g\cap S|(-\log\rho(H^g,S\cap H^g))d\mu_{2n}(g)$. Then
\begin{equation}
\log \rho(H\backslash G,S)-\log\rho(G,S)\geq \frac{I(H,S)}{|S|^2\rho(G,S)^2}.
\end{equation}
\end{zeorem}
\begin{proof}
We follow closely the strategy from \cite{AGV14}. For a walk $w\in S^n$ we will write $T(w)=\{t\in\{1,\ldots,n\}| H w(t-1)=H w(t)\}$. It is the set of times where a walks doesn't change the $H$-coset.  For each walk we define its class $C(w)$ as $$C(w)=\{w'\in S^n| T(w')=T(w) \textrm{ and } w'(t-1)^{-1}w'(t)=w(t-1)^{-1}w(t) \textrm{ for } t\not\in T(w)\}.$$ Two walks are in the same class if they follow the same trajectories on $H\backslash G$ and every time they change $H$-coset they move by the same element of $S$.
For every walk $w\in A(n,S)$  have $C(w)\in A_H(n,S)$ so
$$|A_H(n,S)|\geq \sum_{w\in A(n,S)}\frac{|C(w)|}{|C(w)\cap A(n,S)|}.$$
By \cite[Lemma 9]{AGV14} and the argument from \cite[Lemma 21]{AGV14} we have 
$$\frac{|C(w)|}{|C(w)\cap A(n,S)|}\geq \prod_{t\in T(w)}\rho(H^{w(t)},S\cap H^{w(t)})^{-1}.$$
Now is when we diverge from the argument from \cite{AGV14}. Using the inequality between arithmetic and geometric means we get 
\begin{align*}\frac{|A_H(n,S)|}{|A(n,S)|}\geq &\frac{1}{|A(n,S)|}\sum_{w\in A(n,S)}\prod_{t\in T(w)}\rho(H^{w(t)},S\cap H^{w(t)})^{-1}\\ \geq &\left(\prod_{w\in A(n,S)}\prod_{t\in T(w)}\rho(H^{w(t)},S \cap H^{w(t)})^{-1}\right)^{\frac{1}{|A(n,S)|}}.
\end{align*} Take logarithms of both sides
\begin{align*}\log|A_H(n,S)|-\log|A(n,S)|\geq & \frac{1}{|A(n,S)|}\sum_{w\in A(n,S)}\sum_{t\in T(w)}-\log\rho(H^{w(t)},H^{w(t)}\cap S)\\
= &\frac{1}{|A(n,S)|}\sum_{t=1}^n\sum_{w\in A(n,S)}-\mathbf{1}_{T(w)}(t)\log\rho(H^{w(t)},H^{w(t)}\cap S).\\
\end{align*} We can estimate the rightmost sum by counting for each $t\in\{2,\ldots,n\}$ only the walks of form $w=(s_1,\ldots,s_{t-2},h,h^{-1},s_{t+1},\ldots,s_n)$ with $h\in H^{w(t-2)}\cap S$ and $(s_1,\ldots,s_{t-2},s_{t+1},\ldots,s_n)\in A(n-2,S)$. Thus, for $t\in \{2,\ldots,n\}$ we have 
\begin{align*}
-\sum_{w\in A(n,S)}\mathbf{1}_{T(w)}(t)\log\rho(H^{w(t)},H^{w(t)}\cap S)\geq -\sum_{w\in A(n-2,S)}|H^{w(t-2)}\cap S|\log\rho(H^{w(t-2)},H^{w(t-2)}\cap S)
\end{align*}
We plug it into our previous estimate to get
\begin{align*}
\log|A_H(n,S)|-\log|A(n,S)|\geq &\frac{-1}{A(n,S)}\sum_{w\in A(n-2,S)}\sum_{t=1}^{n-2}|H^{w(t)}\cap S|\log\rho(H^{w(t)},H^{w(t)}\cap S)\\
=& \frac{-(n-2)|A(n-2,S)|}{|A(n,S)|}\int |H^{w(t)}\cap S|\log\rho(H^{w(t)},H^{w(t)}\cap S)d\mu_{n-2}(g)\\
\end{align*}
We divide both sides by $n$ to get 
\begin{equation*}
\frac{\log|A_H(n,S)|}{n}-\frac{\log|A(n,S)|}{n}\geq \frac{-(n-2)|A(n-2,S)|}{n|A(n,S)|}\int |H^{w(t)}\cap S|\log\rho(H^{w(t)},H^{w(t)}\cap S)d\mu_{n-2}(g)\\
\end{equation*}
Replace $n$ by $2n$ and take limes superior of both sides as $n\to\infty$
\begin{equation*}
\log\rho(H\backslash G,S)-\log\rho(G,S)\leq \limsup_{n\to\infty}\frac{-|A(2n-2,S)|}{|A(2n,S)|}\int |H^{w(t)}\cap S|\log\rho(H^{w(t)},H^{w(t)}\cap S)d\mu_{2n-2}(g)\\
\end{equation*}
It remains to estimate the fraction on the right side of the integral. Let $P:l^2(G)\to l^2(G)$ be the transition operator of the random walk on $\Cay(G,S)$. Then $$|A(2n,S)|=|S|^{2n}\langle P^{2n}\mathbf{1}_e,\mathbf{1}_e\rangle=|S|^{2n}\|P^n\mathbf{1}_e\|_{2}^2\leq |S|^{2n}\|P\|^2\|P^{n-2}\mathbf{1}_e\|_{2}^2=|S|^2\rho(G,S)^2|A(2n-2,S)|$$
Hence 
\begin{equation*}
\log\rho(H\backslash G,S)-\log\rho(G,S)\geq\frac{\limsup_{n\to\infty}\int |H^{w(t)}\cap S|(-\log\rho(H^{w(t)},H^{w(t)}\cap S))d\mu_{2n}(g)}{|S|^2\rho(G,S)^2}
\end{equation*}
\end{proof}
\section{Application to uniformly recurrent subgroups}\label{S2}

In this section we prove Theorem \ref{ThmA}.
For the proof define the probability measures $\nu_n$ as 
$$\nu_n=\frac{1}{n}\sum_{k=n+1}^{2n}\mu_{2k}.$$ It has the advantage of being asymptotically quasi invariant with respect to the action of $G$.
\begin{zemma}\label{l.QINV} For any $s\in S$ and any subset $A$ of $G$ we have 
$$\nu_n(As)\geq \frac{n\nu(A)}{(n+1)|S^2|\rho(G,S)^2}-\frac{1}{n}.$$
\end{zemma}
\begin{proof}
First, let us show that for each $s\in S$ we have $$\mu_{2k}(As)\geq \frac{k}{(k+1)|S^2|\rho(G,S)^2}\mu_{2(k+1)}(A).$$ By definition
$$\mu_{2(k+1)}(As)=\frac{1}{(2k+2)|A(2k+2,S)|}\sum_{t=1}^{2k+2}\sum_{w\in A(2k+2,S)}\delta_{w(t)}(As).$$ We estimate the leftmost sum from below by counting the walks $(a_1,\ldots,a_{t-1},s,s^{-1},a_{t+2},\ldots, a_{2k+2})$ with $(a_1,\ldots,a_{t-1},a_{t+2},\ldots,a_{2k+2})\in A(2k,S)$. If $(a_1,\ldots,a_{t-1},a_{t+2},\ldots,a_{2k+2})$ visits $A$ at time $t-1$ then $(a_1,\ldots,a_{t-1},s,s^{-1},a_{t+2},\ldots, a_{2k+2})$ visits $As$ at time $t$. Hence, for $2\leq t \leq 2k+1$
$$\sum_{w\in A(2k+2,S)}\delta_{w(t)}(As)\geq \sum_{w\in A(2k,S)}\delta_{w(t-1)}(A).$$
We get:
\begin{align*}\mu_{2k+2}(As)&\geq \frac{1}{(2k+2)|A(2k+2,S)|}\sum_{w\in A(2k,S)}\sum_{t=2}^{2k+1}\sum_{w\in A(2k,S)}\delta_{w(t-1)}(A)\\ &=\frac{(2k)|A(2k,S)|}{(2k+2)|A(2k+2,S)|}\mu_{2k}(A).\end{align*} As 
$$\frac{|A(2k,S)|}{|A(2k+2,S)|}\geq \frac{1}{|S|^2\rho(G,S)^2}$$ we get 
$$\mu_{2k}(As)\geq \frac{k\mu_{2(k+1)}(A)}{(k+1)|S^2|\rho(G,S)^2}.$$ It follows that 
\begin{align*}
\nu_{n}(As)+\frac{\mu_{2n+2}(As)}{n}\geq & \frac{n\nu_n(A)}{(n+1)|S^2|\rho(G,S)^2}\\
\nu_{n}(As)\geq \frac{n\nu_n(A)}{(n+1)|S^2|\rho(G,S)^2}-\frac{1}{n}.
\end{align*}
\end{proof}
\begin{proof}[Proof of Theorem \ref{ThmA}]
 Let $H_0$ be a uniformly recurrent subgroup such that $\rho(H_0\backslash G,S)=\rho(G,S)$. By Theorem \ref{t.SpecIneq} we have $\limsup_{n\to\infty}\int |H_0^g\cap S|(-\log\rho(H_0^g,H_0^g\cap S)d\mu_{2n}(g)=0$. Then we also have 
$$\lim_{n\to\infty}\int |H_0^g\cap S|(-\log\rho(H_0^g,H_0^g\cap S)d\nu_{n}(g)=0.$$
Let $\delta_{H_0}$ be the dirac mass in $H_0$ and let $\omega$ be a weak-* limit of measures $\delta_{H_0}*\nu_n$ as $n\to\infty$. Then 
$$\int |H\cap S|(-\log\rho(H,H\cap S))d\omega(H)=0$$
so the set of $H$ such that $\log\rho(H,H\cap S)=0$ has full measure. By Kesten's criterion this is precisely the set of $H$ for which the group $\langle H\cap S\rangle $ is amenable. As the set $S$ is finite, the subset of $H\in \Sub_G$ such that $\langle H\cap S\rangle$ is not amenable is open. 

From Lemma \ref{l.QINV} we deduce that $\omega$ is quasi-invariant, i.e. $\omega(E)=0$ if and only if $\omega(gE)=0$ for all $g\in G$. The support of $\omega$ is a closed $G$-invariant subset of $X$ so by minimality it has to be the whole $X$. In particular any non-empty open set has positive measure. It follows that $\langle H\cap S\rangle$ is amenable for all $H\in X$. By taking $S'=(\{1\}\cup\{S\})^m$ and letting $m$ go to infinity we show in this way that every $H\in X$, including $H_0$,  is amenable.
\end{proof}
\begin{zorollary}
Let $G$ be countable $C^*$-simple group with a finite symmetric generating set $S$. If $H$ is a Ramanujan subgroup of $G$ then there exists a sequence $(g_i)_{i\in\N}$ such that $H^{g_i}$ converges to the identity subgroup in the Chabauty topology.
\end{zorollary}
\begin{proof}
The closure $X=\overline{\left\{H^g\mid g\in G\right\}}$ consists of Ramanujan subgroups. By Zorn's lemma there exists a minimal $G$-invariant closed subset $Y\subset X$. By Theorem \ref{ThmA} it consists of uniformly recurrent amenable sugbroups and Kennedy's criterion \cite{Kennedy} yields $Y=\{1\}$ which proves the assertion. 
\end{proof}
\section{Cycle density along random walks}\label{S3}
Let $\mathcal G$ be a $d$-regular graph. For any vertex $x$ and $k\geq 1$ put $C_G(x,k)=1$ if there exists a non-backtracking cycle of length $k$ starting at $x$ and $C_G(x,k)=0$ otherwise. We say that two cycles in $\mathcal G$ are independent if they generate non-abelian free subgroup of the fundamental group of $\mathcal G$. Let $D_{\mathcal G}(x,k)=1$ if there exist at least two independent non-backtracking cycles of length $k$ starting at $x$ and $D_{\mathcal G}(x,k)=0$ otherwise. In this section we prove:
\begin{zeorem}\label{T1S3}
Let $G$ be $d$-regular rooted Ramanujan graph. Let $(X_i)$ be the standard random walk on $G$. Then for any $k\geq 1$ 
$$\lim_{n\to\infty}\frac{1}{n}\sum_{i=1}^n \Eb{C_G(X_i,k)}=0.$$
\end{zeorem}
Write $\mathbb T_d$ for the $d$-regular rooted tree. If $\mathcal G=(V,E)$ we shall write $\mathcal G^{(k)}$ for the graph vertex set $V$ with multiple edges where the number of  edges between  $v_1,v_2)\in V\times V$  is given by $$w(v_1,v_2)=|\{(e_0,e_k)| (e_0,e_1,\ldots,e_k) \textrm{ is a non-backtracking 
walk in } \mathcal G\}|.$$ Graph $\mathcal{G}^{(k)}$ is always a $d(d-1)^{k-1}$-regular graph and $C_{\mathcal G}(x,k)=C_{\mathcal G^{(k)}k}(x,1)$. We have
\begin{zemma}\label{L2S3}
$\mathcal G$ is a Ramanujan graph if and only $\mathcal G^k$ is. 
\end{zemma}
\begin{proof}
Since $\mathbb T_d^k=T_{d(d-1)^{k-1}}$ it enough to observe that $\rho(\mathcal G^k)$ is a strictly decreasing function of $\rho(\mathcal G)$. 
\end{proof}

We will use the notion of a \textbf{stationary random graph}. We think of the $d$-regular random graphs as the Borel probability measures on the space of isomorphism classes of rooted $d$-regular graphs. For more comprehensive introduction to random graphs we refer to \cite{AGV12}. A random, rooted $d$-regular graph $(\tilde{\mathcal G},\tilde{x})$ is called stationary if its probability distribution is invariant under replacing the root $\tilde x$ by a random neighbor. If $H$ is a random subgroup of a group $G$ satisfying $\Eb{f(H)}=\frac{1}{|S|}\sum_{s\in S}\Eb{f(H^s)}$ for every continuous function $f$ on $\Sub_G$ then we will call $H$ a \textbf{stationary random subgroup}. The Schreier graph $\Sch(H\bs G,S)$ rooted at the identity coset is an example of a stationary random graph. 
\begin{lemma}\label{L3S3}
Let $(\mathcal{G},o)$ be a $d$-regular stationary random graph, then $(\mathcal{G}^k,o)$ is a $d(d-1)^{k-1}$-regular stationary random graph.
\end{lemma}\begin{proof}Let $P,P^{(k)}$ be the averaging operators associated to the standard random walks on $\mathcal G,\mathcal G^{(k)}$ respectively. Since $(\mathcal G,o)$ is stationary the operator $P$ fixes its probability distribution. The operator $P^{(k)}$ is a polynomial in $P$ so it also fixes the distribution of $(\mathcal{G}^{(k)},o)$.
\end{proof}
\begin{proof}[Proof of Theorem \ref{T1S3}]
We argue by contradiction. We start with a Ramanujan graph and a natural number $k$ for which the conclusion does not hold. In the first step we construct a stationary random Ramanujan graph where the root is contained in a $k$-cycle with positive probability. In the second step we use Lemma \ref{L2S3} to upgrade it to a stationary random Ramanujan graph of even degree where the root is contained in two independent cycles of length $ak$ with positive probability. Regular graphs of even degree are Schreier graphs on cosets of a free group, this allows to reduce our problem to Theorem \ref{t.SpecIneq} in steps 3,4 and 5. 

\textbf{ Step 1.} We replace $\mathcal G$ by a stationary random graph. Let 
$$0<\alpha=\limsup_{n\to\infty}\frac{1}{n}\sum_{i=1}^n \Eb{C_{\mathcal G}(X_i,k)}.$$ There exists an increasing sequence $(n_i)_{i\in \N}$ such that $\alpha=\lim_{i\to\infty}\frac{1}{n_i}\sum_{j=1}^{n_i} \Eb{C_{\mathcal G}(X_j,k)}.$
For each $i\geq 0$ let $(\mathcal G,X_i)$ be a random rooted graph where root is given by the position of the random walk at time $i$. Let $(\tilde{\mathcal G},\tilde{x})$ be any weak limit of the sequence of random graphs $$\frac{1}{n_i}\sum_{j=1}^{n_i} (\mathcal G,X_j).$$ Then $(\tilde{\mathcal G},\tilde x)$ is a stationary random graph and $\Eb{C_{\tilde{\mathcal G}}(\tilde{x},k)}=\alpha>0$.  Moreover, since $\mathcal G$ was Ramanujan $\tilde{\mathcal G}$ is Ramanujan almost surely.

\textbf{Step 2.} We show that there exists $a>1$ such that $\Eb{D_{\tilde{\mathcal G}^{(a)}}(\tilde{x},k)}>0$. Since $D_{\tilde{\mathcal G}^{(a)}}(\tilde{x},k)=D_{\tilde{\mathcal G}^{(ka)}}(\tilde{x},1)$ we may assume in this step that $k=1$. First, we claim that for $A$ big enough the ball $B_{\tilde{\mathcal G}}(\tilde x, A)$ contains at least two vertices $y_1,y_2$ such that  $C_{\tilde{\mathcal G}}(y_1,1)=C_{\tilde{\mathcal G}}(y_2,1)=1$. Let $(\tilde{X}_i)$ denote the random walk on $\tilde{\mathcal G}$. By stationarity we have $(\tilde{\mathcal G},\tilde{x})=(\tilde{\mathcal G},\tilde{X_i})$ for all $i\in\N$. Since $\tilde{\mathcal G}$ is Ramanujan almost surely, the $l^\infty$ norm of the probability distribution of $\tilde X_i$ decreases exponentially fast with $i$. Therefore there exists an $A>0$ such that for almost all $(\tilde{\mathcal G},\tilde x)$ and for every vertex $v\in \tilde{\mathcal G}$ we have $\Pb{\tilde X_A=v}<\frac{\alpha}{3}.$  Then, the equality $\Eb{C_{\tilde{\mathcal G}}(\tilde{X}_A,k)}=\Eb{C_{\tilde{\mathcal G}}(\tilde{x},k)}=\alpha$ implies that with positive probability $(\tilde{\mathcal G},\tilde x)$ is such that  there are at least $2$ possible values for $\tilde{X}_A$ where $C(\tilde{X}_A,1)=1$. That proves the claim. If the ball  $B_{\tilde{\mathcal G}}(\tilde x, A)$ contains two distinct vertices with loops, then we can construct two independent non-backtracking cycles of length $2A+1$ starting at $\tilde x$. Thus $\Eb{D_{\tilde{\mathcal G}^{(2A+1)}}(\tilde{x},1)}>0$. It is enough to take $a\geq 2A+1$. By Lemma \ref{L3S3} the random graph $(\tilde{\mathcal G}^{(a)},\tilde x)$ is a stationary random graph.

\textbf{Step 3.} Put $2d'=d(d-1)^{a-1}$, $d'$ is an integer since $a\geq 2$. We use the random graph $(\tilde{\mathcal G}^{(a)},\tilde x)$ to construct a random Schreier graph $\Sch(H\backslash F_{d'},S)$ where $F_{d'}$ is the free group on $d'$ generators, $S$ the standard symmetric free generating set and  $H$ a stationary random subgroup of $F_{d'}$. Let $\tilde Y_i$ be the standard random walk on $(\tilde{\mathcal G}^{(a)},\tilde x)$. This is a different walk than $\tilde X_i$ since we have modified the edge set. By Fatou lemma and stationarity  $$\Eb{\limsup_{n\to\infty}\frac{1}{n}\sum_{i=1}^n D_{\tilde{\mathcal G}^{(a)}}(\tilde Y_i,k)}\geq \Eb{D_{\tilde{\mathcal G}^{(a)}}(\tilde x,k)}>0.$$
In particular there exist (deterministic) $d'$-regular Ramanujan graph $(\mathcal G_1,x_1)$ such that \begin{equation}\label{EQ1S3}\limsup_{n\to\infty}\frac{1}{n}\sum_{i=1}^n D_{\tilde{\mathcal G_1}}(\tilde Y_i,k)>0.\end{equation}
 Note that the degree of $\mathcal G_1$ is even so by \cite{Gross} it is isomorphic to a Schreier graph. Hence, there exists a subgroup $H_0\subset F_{d'}$ such that $(\mathcal G_1,x_1)\simeq \Sch(H_0\backslash F_{'d'},S)$. We construct a stationary random subgroup $H$ as a weak-* limit of $\frac{1}{n}\sum_{i=1}^n\frac{1}{|S|^i}\sum_{s\in S^i}H^s$ along a subsequence for which the limit superior (\ref{EQ1S3}) converges to a positive number. Then we have $\Eb{D_{\Sch(H\backslash F_{'d'},S)}(H,k)}>0$.
 
 \textbf{Step 4.} We reinterpret the condition $\Eb{D_{\Sch(H\backslash F_{d'},S)}(H,k)}>0$ in terms of the expected spectral radius. Let $H_1$ be any deterministic subgroup of $F_{d'}$. Any two independent non-backtracking $k$-cycles $c_1,c_2$ in $\Sch(F_{d'}/H_1,S)$ starting at $H_1$ give rise to two elements $a,b\in S^k\cap H$ generating freely a free subgroup. Hence there is $\beta=\beta(k,d')<1$ such that $D_{\Sch(F_{'d'}/H_1,S)}(H_1,k)>0$ implies $\rho(H_1 ,H_1\cap S^k)\leq \beta$. We deduce that $\Eb{-\log\rho(H_1,H_1\cap S^k)}>0$. 
 
 \textbf{Step 5.} We use Theorem \ref{t.SpecIneq} to get a contradiction. The graph $\Sch(F_{d'}/H,S)$ is Ramanujan almost surely so by Theorem \ref{t.SpecIneq}
\begin{equation}\label{ineq3}\lim_{n\to\infty}-\EB{\int_{F_{d'}}|H^g\cap S^k|\log\rho(H^g, H^g\cap S^k\rangle)d\mu_{2n}(g)}=0.
\end{equation}
The density function of $\mu_{2n}$ for the free group and the standard symmetric generating set is a spherical function on $F_{d'}$ (its value depends only on the distance from the root). Hence, we can use the property that $H$ is stationary to get 
$$-\EB{\int_{F_{d'}}|H^g\cap S^k|\log\rho(H^g, H^g\cap S^k\rangle)d\mu_{2n}(g)}=-\Eb{|H\cap S^k|(-\log\rho(H_1,H_1\cap S^k))},$$ which together with (\ref{ineq3}) contradicts the conclusion of the fourth step.
\end{proof}
Using the same reasoning we can show 
\begin{theorem}\label{tstationary}
Let $(\mathcal{G},o)$ be stationary random $d$-regular graph. If $(\mathcal{G},o)$ is Ramanujan almost surely then it is a $d$-regular tree. 
\end{theorem}
\begin{proof}
We argue by contradiction. Let $(\mathcal{G},o)$ be a stationary random $d$-regular graph which is almost surely Ramanujan but is not a tree with positive probability. Then, there exists $k$ such that $\Eb{C_{\mathcal G}(o,k)}=\alpha>0$. Let $(X_n)_{n\in\N}$ be the nearest-neighbor random walk on $\mathcal G$. Since $(\mathcal G,o)$ is stationary we have $\Eb{C_{\mathcal{G}}(X_n,k)}=\Eb{C_{\mathcal G}(o,k)}=\alpha$. By standard application of Fatou lemma we get that 
$$\Eb{\limsup_{i\to\infty}\frac{1}{n_i}\sum_{j=1}^{n_i} C_{\mathcal G}(X_j,k)}\geq \alpha.$$ As the graph $(\mathcal G,o)$ is Ramanujan almost surely we deduce that $(\mathcal{G},o)$ is a counter-example for Theorem \ref{T1S3} with positive probability. This gives the desired contradiction. 
\end{proof}
\bibliographystyle{plain}
\bibliography{ref}
\end{document}